\documentclass[12pt, twoside, leqno]{article}
\usepackage{amsmath,amsthm}
\usepackage{amssymb,latexsym}
\usepackage{enumerate}

\newtheorem{thm}{Theorem}[section]

\newtheorem{lem}[thm]{Lemma}

\theoremstyle{definition}
\newtheorem{defin}[thm]{Definition}
\newtheorem{rem}[thm]{Remark}
\newtheorem{exa}[thm]{Example}

\numberwithin{equation}{section}

\begin{document}

\baselineskip=17pt

\title{Spectral Factorization of Trigonometric Polynomials and Lattice Geometry}

\author{Wayne Lawton\\
Department of Mathematics\\
Mahidol University, Bangkok, Thailand \\
and School of Mathematics and Statistics\\
University of Western Australia, Perth, Australia\\
E-mail: wayne.lawton@uwa.edu.au}

\date{}

\maketitle

%% Classification and key words; note that the 2010 classification is used:

\renewcommand{\thefootnote}{}



\renewcommand{\thefootnote}{\arabic{footnote}}
\setcounter{footnote}{0}

%%%%%%%%

\begin{abstract}
We formulate a conjecture concerning spectral factorization
of a class of trigonometric polynomials of two variables
and prove it for special cases. Our method uses relations
between the distribution of values of a polynomial of two variables
and the distributions of values of an associated family of polynomials of
one variable. We suggest an approach to prove the full conjecture
using relations between the distribution of values and the distribution
of roots of polynomials.
\end{abstract}
\section{Introduction}
The Fej\'{e}r-Riesz spectral factorization lemma (\cite{Riesz-Nagy}, p. 117),
conjectured by Fej\'{e}r \cite{Fejer} and proved by Riesz \cite{Riesz}, shows
that if $F = \{0,...,n\}$ then every nonnegative trigonometric polynomial
of one variable having frequencies in the difference set $F - F =  \{-n,...,n\}$ equals the squared modulus of a trigonometric polynomial having frequencies in $F.$ This result, and partial extensions to trigonometric polynomials of several variables, has an enormous range of applications in
analysis and engineering. It also has significant, though less acknowledged, connections with number theory. We review one of these. If $q$ is a monic polynomial with all roots having moduli $\geq 1,$ then
%
% page 2
%
Jensen's theorem implies that the product of the moduli of the roots of $q$ equals the geometric mean of $p = |q|^2$ over the circle. This quantity and its extension to several variables, called the Mahler measure or height of a polynomial, has applications to Lehmer's problem \cite{Lehmer}, transcendental numbers \cite{Waldschmidt}, and algebraic dynamics \cite{Everest}. Boyd \cite{Boyd} and myself \cite{Lawton1} conjectured that heights of polynomials of several variables were limits of heights of associated polynomials of one variable and showed that the validity of the conjecture
would provide an alternative proof of a special case of Lehmer's problem proved by Dobrowolski, Lawton and Schinzel in \cite{Dobrowolski}. In \cite{Lawton2} I derived an upper bound for the measure of the set where a monic polynomial of one variable can have a small modulus and used it to prove this conjecture. In \cite{Schinzel} Schinzel used this bound to derive an inequality for the Mahler measure of polynomials in many variables.
\\ \\
In this paper we also use the relationships between distributions of values of polynomials in two variables and associated families of polynomials in one variable. Consider the set of lattice points $F$ in the open planar region bounded by lines $y = \alpha x \pm \beta$ and let $\mathfrak{U}_2(F)$ be the set of trigonometric polynomials of variables $x$ and $y$ that are uniform limits of squared
moduli of trigonometric polynomials whose frequencies are contained in $F.$ The trigonometric polynomials in $\mathfrak{U}_2(F)$ are nonnegative and their frequencies are contained in the difference set $F - F.$ We conjecture that the converse holds. The main results in this paper, derived in Section 3, prove our conjecture for the three cases where $\alpha$ is zero, rational, or can approximated by rational numbers with sufficient accuracy.
The case $\alpha = 0$ is solved by applying Fej\'{e}r-Riesz spectral factorization
to a trigonometric polynomial $P(x,y)$ regarded as a polynomial in $y$ with coefficients that are functions of $x$ and then approximating the spectral factors of $P$ by trigonometric polynomials. The case where $\alpha$ is rational is reduced to the case $\alpha = 0$ by transforming $P$ using an element in the modular group $SL(2,\mathbb{Z}).$ When $\alpha$ is irrational we approximate it by rational numbers but our proof only works for values of $\alpha$ that can be approximated sufficiently rapidly by rationals. We suggest approaches to strengthen these results.
\section{Notation and Preliminary Results}
$\mathbb{Z}_{+}, \mathbb{Z}, \mathbb{Q}, \mathbb{R}, \mathbb{C}$ denote the
nonnegative integer, integer, rational,
real, and complex numbers,
$T = \mathbb{R}/\mathbb{Z}$ and $T_c = \{z\in \mathbb{C}:|z|=1\}$
%
% page 3
%
are the real and complex circle groups.
For $j \in \mathbb{Z}$ we define $e_j : T \rightarrow T_c$ by
$e_j(x) = e^{2\pi i j x}.$
For $z = x+iy, x,y \in \mathbb{R},$ $\Re\, z = x$ and $\Im\, z = y.$
For $z = re^{i\theta}$ with $r >0$ and $\theta \in (-\pi,\pi],$
$\log (z) = \log (r)+ i \theta$ and
$\sqrt z = {\sqrt r} e^{i\theta/2}.$
For $p \in [1,\infty) \cup \{\infty\},$ $\ell^p(\mathbb{Z}), L^{p}(T)$ are the Banach spaces
with norms $||f||_{p}.$
If $f \in L^1(T)$ and $h \in L^{p}(T)$ then their \emph{convolution}
$(f*h)(x) = \int_{y\in T} f(y)h(x-y)dx$
is in $L^{p}(T)$ and
$||f*h||_p \leq ||f||_1\, ||h||_p.$
For $f \in L^{1}(T)$ the \emph{Fourier transform}
${\widehat f} : \mathbb{Z} \rightarrow \mathbb{C}$
is defined by
${\widehat f}\, (j) = \int_{x \in T} f(x) \, e_{-j}(x) \, dx.$
If $f, h \in L^1(T)$ then
$\widehat {f*h} = {\widehat f}\ {\widehat h}.$
The set
$\hbox{freq}(f) = \hbox{support}(\widehat f)$
satisfies $\hbox{freq}(|f|^2) \subseteq \hbox{freq}(f) - \hbox{freq}(f).$
\begin{lem}
\label{kernels}
For $N \geq 0$ the
\emph{Dirichlet} kernel
$D_N = \sum_{j=-N}^{N} e_j,$
\emph{Hilbert} kernel
$H_N = \sum_{j=1}^{N} (e_j - e_{-j}),$
and \emph{analytic} kernel
$A_N^{\pm} = \frac{1}{2}(D_N \pm H_N)$
satisfy
$$
    ||D_N||_{1} \leq 1 + \log (2N+1), \ \
||A_N^{\pm}||_{1} \leq \frac{3}{2} + \log (N),
$$
$$
    ||1/2 + A_N^{\pm}||_{1} \leq 1 + \log (N+1), \ \mbox{and} \
    ||H_N||_{1} \leq 1 + 2\log (N).
$$
\end{lem}
\begin{proof}
Follows from the derivations in {\cite[Section 16.2]{Powell}}.
\end{proof}
\noindent For a topological space $X,$ $C(X)$ denotes the Banach algebra of bounded
continuous functions $f:X\rightarrow \mathbb{C}$ with norm
$||f||_{\infty} = \sup \, \{\, |f(x)| : x \in X\, \}.$
The \emph{Wiener algebra}
$A(T) = \{ f \in C(T) : \widehat f \in \ell^1(\mathbb{Z}) \}$
and its subalgebras
$A^{\pm}(T) = \{\, f \in A(T) \, : \, \hbox{freq}(f) \subseteq \pm \mathbb{Z}_+ \, \}$
are Banach algebras with norm
$||f||_{A(T)} = ||\widehat f||_1$
so they are closed under exponentiation.
\begin{lem}
\label{truncate}
If $f \in A^{\pm}(T)$ and $0 \leq n \leq N$ then
$$
    (1/2+A_{n}^{\pm})*\exp(f) = (1/2+A_{n}^{\pm})*\exp((1/2+A_{N}^{\pm})*f).
$$
\end{lem}
\begin{proof}
$\exp(f - (1/2+A_{N}^{\pm})*f) = 1 + \sum_{j=N+1}^{\infty} c(j) \, e_{\pm j}$ with $c \in \ell^1(\mathbb{Z})$
so
$\exp(f) = \exp((1/2+A_{N}^{\pm})*f) + \sum_{j=N+1}^{\infty} d(j) \, e_{\pm j}$ with $d \in \ell^1(\mathbb{Z})$
and hence Lemma \ref{truncate} follows since $(1/2+A_{n}^{\pm})*e_j = \chi_{\pm \{0,1,...,N\}}(j) \, e_j.$
\end{proof}
\noindent If $B$ is a commutative Banach algebra $B$ then
$\exp(B) = \{\exp(b):b \in B\}$
is an open and closed subgroup of the group
$\hbox{Inv}(B)$
of invertible elements.
For
$f \in \hbox{Inv}(C(T))$
we define the \emph{winding number}
$W(f) \in \mathbb{Z}$ by
$$
    W(f) = \lim_{M \rightarrow \infty} \sum_{j=0}^{M-1} \log \left[ \frac{f((j+1)/M)}{f(j/M)} \right].
$$
%
% page 4
%
\begin{lem}
\label{WAR}
For $f \in \hbox{Inv}(C(T))$ with $W(f) = 0$  define $L(f) \in C(T)$ by
$$
    L(f)(x) = c + \lim_{M \rightarrow \infty} \sum_{j=0}^{M-1} \log \left[ \frac{f((j+1)x/M)}{f(jx/M)}\right]
$$
where $\exp(c) = f(0).$ Then $L(f)$ is unique up to addition by an element in $2\pi i\mathbb{Z},$
$\exp(L(f)) = f,$ and
$\exp(C(T)) = \{f \in Inv(C(T)): W(f) = 0\}.$
$Inv(A(T)) = A(T) \cap Inv(C(T))$ and
$\exp(A(T)) = A(T) \cap \exp(C(T)).$
\end{lem}
\begin{proof}
The assertions about $C(T)$ follow directly. The characterization of $A(T)$ follows from
Wiener's Tauberian lemma \cite{Wiener}. The characterization of $\exp(A(T))$ follows from
the Arens-Royden theorem \cite{Arens}, \cite{Royden}.
\end{proof}
\noindent Define
$A_{\infty}^{\pm} : A(T) \rightarrow A^{\pm}(T)$ by
$A_{\infty}^{\pm}(f) = \frac{1}{2}{\widehat f}(0) + \sum_{j=1}^{\infty} {\widehat f}(\pm j)\, e_{\pm j},$
and define $\Psi^{\pm} : \exp(A(T)) \rightarrow \exp(A^{\pm}(T))$ by
$\Psi^{\pm}(f)= \exp (A_{\infty}^{\pm}(L(f))).$
The $\Psi^{\pm}(f)$ are uniquely defined up to a multiple of $\pm 1,$
$\Psi^{+}(f)\Psi^{-}(f) = f,$
and if $f > 0$ then
$\Psi^{-}(f) = \overline {\Psi^{+}(f)}.$
$\mathfrak{T}_1$ is the algebra of \emph{trigonometric polynomials}, and for $F \subset \mathbb{Z},$
$
    \mathfrak{T}_1(F) =  \{ f \in \mathfrak{T}_1 : \hbox{freq}(f) \subseteq F\}.
$
Define $n^{\pm}, n \, :  \, \mathfrak{T}_1 \rightarrow \mathbb{Z} \cup \{\infty\}$ by
$n^{\pm}(0) = n(0) = \infty$ and
for nonzero $t,$ $n^{\pm}(t) = \max (\{0\} \cup \pm \hbox{freq}(t)),$ $n(t) = \max\{n^{+}(t),n^{-}(t)\}.$
The following result follows easily:
\begin{lem}
\label{poly}
If $t \in \mathfrak{T}_1$ and $|t| > 0$ and $W(t) = 0$ then there exist nonzero complex numbers
$\lambda_{1}^{-},...,\lambda_{n^{-}(t)}^{-},\lambda_{1}^{+},...,\lambda_{n^{+}(t)}^{+}$ with moduli $< 1$
such that
$$
\Psi^{\pm}(t) = e^{\gamma(t)/2} \, \prod_{j=1}^{n^{\pm}(t)} (1 - \lambda_{j}^{\pm} e_{\pm 1}),
$$
where
$\gamma(t) = c - \sum_{j=1}^{n^{-}(t)} \log (1-\lambda_{j}^{-}) - \sum_{j=1}^{n^{+}(t)} \log (1-\lambda_{j}^{+})$ and $\exp (c) = t(0).$
\end{lem}
\begin{lem}
\label{Psibound1}
If $t \in \mathfrak{T}_1 \cap \exp(A(T))$ and $N \geq n^{\pm}(t)$ then
$$
    ||\Psi^{\pm}(t)||_{\infty} \leq
    c_{n(t)} \, \exp \left[ \, \frac{1}{2}\max \, (\, D_N * \log (|t|) + H_N*\Im L(t) \, ) \, \right]
$$
where $c_{n} =(1+\log (n+1)).$
\end{lem}
\begin{proof}
Lemma \ref{poly} gives
$\Psi^{\pm}(t) = (\frac{1}{2} + A_{N}^{\pm})*\exp (A_{\infty}^{\pm}(L(f))).$
Then since
$A_{N}^{\pm}*L(t) = (\frac{1}{2} + A_{N}^{\pm})*A_{\infty}^{\pm}(L(t)),$
Lemma \ref{truncate} gives
$\Psi^{\pm}(t) = (\frac{1}{2} + A_{N}^{\pm})*\exp (A_{N}^{\pm}*L(t)),$
so Lemma \ref{kernels} gives
$||\Psi^{\pm}(t)||_{\infty} \leq c_n \, \exp (\max \Re (A_{N}^{\pm}*L(t))).$
Finally substitute
$\Re (A_{N}^{\pm}*L(t)) = (D_{N}*\log (|t|) + iH_{N}*\Im L(t))/2.$
\end{proof}
\begin{exa}
Lemma \ref{Psibound1} with $N = n^{\pm}(t)$ shows that
$||\Psi^{\pm}(t)||_{\infty}$ is bounded by a polynomial
in $n^{\pm}(t).$ The degree may be large.
For odd $n \in \mathbb{Z}_+$ let
$P_n(z) = (z-1/n)^{2n} -1$ and $t_n = e_{-n}P_n(e_1).$
Then $n^{\pm}(t_n) = n$ and
$$
\Psi^{\pm}(t_n) = \sqrt {M(t_n)} \, \prod_{k=-\ell}^{\ell}  \left[ \, 1-\left( \, \frac{1}{n} \pm e^{\pi i k/n} \, \right)^{\mp 1} \, e_{\pm 1} \, \right],
$$
%
% page 5
%
where $\ell = (n-1)/2$ and $M(t_n)$ is the Mahler measure of $t_n.$ For large $n,$
$M(t_n) \approx \exp(2/\pi),$
$||t_n||_{\, \infty} \approx 1+e^2,$
$||\Im \, L(t_n)||_{\, \infty} \approx 2\, \pi \, n,$
and
$$
\log \, ||\Psi^{\pm}(t_n)||_{\, \infty} \approx \frac{1}{\pi} + 2n\int_{-\frac{1}{4}}^{\frac{1}{4}} \log \, (\, |1+e^{2 \pi i x}\, |\, ) \, dx \approx \frac{1}{\pi}+0.5831 \, n.
$$
If $\Re \, t > 0$ then
$L(t) = \log (t) = \log(|t|) + i \Im L(t)$
and
$||\Im L(t)||_{\infty} <  \pi/2.$
We will use this inequality to derive much smaller upper bounds for
$||\Psi^{\pm}(t_n)||_{\, \infty}.$
\end{exa}
\noindent For $X \subset \mathbb{C},$ $X^{o}$ is its interior and
$H(X)$ is the Banach subalgebra of $C(X)$ consisting of functions whose restriction to $X^{o}$ is
holomorphic.
For $\sigma > 0,$
$
    A_\sigma = \{ z \in \mathbb{C}: e^{-\sigma} \leq |z| \leq e^{\sigma} \}
$
is a closed annulus. $C^{\omega}(T)$ is the algebra of \emph{real analytic} functions on $T.$
{\cite[Chapter 10, Exercise 24 ]{Rudin1}} gives
\begin{lem}
\label{decay}
$f \in C^{\omega}(T)$ iff there exist $\sigma > 0$ and $h \in H(A_\sigma)$
such that $f = h(e_1).$ Then for every $N \geq 0,$
$
    ||f - D_N*f||_{\infty} < 2 \, ||h||_{\infty} \, \sigma^{-1} \, e^{-N\sigma}.
$
\end{lem}
\noindent For $t \in \mathfrak{T}_1$ with $\Re \, t > 0$ let
$
   \rho(t) =  \min \Re \, t/(2 e  ||\, \widehat t \, ||_1),
$
$
    \sigma(t) = \rho(t)/n(t),
$
$
    \\ \tau(t) = \max \ \{\log \, (\, ||t||_{\infty}\, ) + \frac{1}{2} \min \Re \, t, \ \frac{\pi}{2} + |\log (\frac{1}{2} \min \Re\, t)|\},
$
and
$
    \\ \theta(t) = \arctan \left( \, \max ||\Im t||_{\infty} / \min \Re t \, \right).
$
We observe that $\theta(t) < \pi/2.$
\begin{lem}
\label{F1}
If $t \in \mathfrak{T}_1$ and $\Re \, t > 0$ then
there exists a unique $h \in H(A_{\sigma(t)})$ such that $\log (t) = h(e_1).$
Furthermore
$
   ||h||_{\infty} \leq \tau(t),
$
and for every $N > 0,$
$
    D_N*\log (|t|) \leq \log (|t|) + 2\tau(t) \sigma(t)^{-1} \exp (-N\sigma(t)).
$
\end{lem}
\begin{proof}
Let $P$ be the Laurent polynomial with $t = P(e_1).$ \\
If $z \in A_{\sigma(t)}$
then $z = e^s \, e_1(x)$ where $s \in \mathbb{R},$ $|s| \leq \sigma(t),$
$x \in T.$ Therefore
$|P(z) - t(x)| \leq \min \Re \, t /2$
so
$\Re \, P(z) \geq \min \Re \, t/2,$
$h = \log (P) \in H(A_{\sigma(t)}),$ and $\log (t) = h(e_1).$
The last assertions follows from Lemma \ref{decay}.
\end{proof}
\begin{thm}
\label{Psibound2} There exists
$\lambda : \left(0,\frac{1}{2e}\right) \times \left(0,\frac{\pi}{2}\right) \times \left(\frac{\pi}{2},\infty\right)
\rightarrow (0,\infty)$
such that if $t \in \mathfrak{T}_1$ and $\Re \, t > 0$ then
$
    ||\Psi^{\pm}\, t||_\infty \leq \lambda(\rho(t),\theta(t),\tau(t)) \, n(t)^{\pi/2}.
$
\end{thm}
\begin{proof}
For every $N \geq n(t),$ Lemma \ref{kernels}, Lemma \ref{Psibound1}, and Lemma \ref{F1} give
$$
    ||\Psi^{\pm}(t)||_{\infty} \leq c_{n(t)} e^{\theta(t)/2} \, ||t||_{\infty}^{1/2} \, N^{\theta(t)} \,
    \exp \left[\frac{n(t)\tau(t)}{\rho(t)} \exp\left(-\frac{N\rho(t)}{n(t)}\right)\right].
$$
Then choose $N = N_0(t) = (n(t)/\rho(t)) \log (n(t)\tau(t)/\rho(t))$ to obtain \\
$||\Psi^{\pm}(t)||_{\infty} \leq c_{n(t)} e^{\theta(t)/2} \, ||t||_{\infty}^{1/2} \, N_{0}(t)^{\theta(t)}$
and use the relation $\theta(t) < \pi/2.$
\end{proof}
\noindent For $d \geq 1,$
$\mathfrak{T}_d$ is the algebra of trigonometric polynomials on $T^d,$
$\mathfrak{T}_d^{+}  = \{t \in \mathfrak{T}_d : \Re \, t > 0\},$ and
$\mathfrak{T}_d^{p}  = \{t \in \mathfrak{T}_d : t > 0\}.$
$C^{\omega}(T) \bigotimes \mathfrak{T}_1$ is the subset of $f \in C(T^2)$ such that
$f(x,y)$ is a trigonometric polynomial in
%
% page 6
%
$y$ whose coefficients are analytic functions
of $x.$ For such a function and $N \geq 0$
we define
$f_N \in \mathfrak{T}_2$ by
$
    f_N(x,y) = \int_{u \in T} D_N(u)\, f(x-u,y) du.
$
For
$t \in \mathfrak{T}_{2}^{+}$
and
$x \in T$
define
$t_x \in \mathfrak{T}_{1}^{+}$
by
$t_x(y) = t(x,y), \ y \in T$
and define spectral factorization operators
$S^{\pm} \, : \, \mathfrak{T}_{2}^{+} \rightarrow C^{\omega}(T) \bigotimes \mathfrak{T}_1$
by
$S^{\pm}(t)(x,y) = \Psi^{\pm}(t_x)(y), \ \ (x,y) \in T^2.$
Then $t = S^{-}(t)S^{+}(t),$ if $t > 0$ then $S^{-}(t) = \overline {S^{+}(t)},$ and
$||S^{\pm}(t) - (S^{\pm}(t))_N||_{\infty} \rightarrow 0.$
The remainder of this section computes the rate of convergence.
For $t \in \mathfrak{T}_{2}$ let $n_1(t), n_2(t)$ denote $n(t)$ where $t(x,y)$ is considered as a polynomial
in $x, y$ respectively.
Extend the definitions
of $\rho(t),$ $\theta(t),$ and $\tau(t)$ (given before Lemma \ref{F1})
to the case where $t \in \mathfrak{T}_{2}^{+}$ and define
$\sigma_1(t) = \rho(t)/n_1(t).$
For $(j,k) \in \mathbb{Z}^{2}$ define
$
    e_{j,k}(x,y) = e_j(x)e_k(y)
$
and for $z \in \mathbb{C}$ define the algebra homomorphism
$\Gamma_z : \mathfrak{T}_2 \rightarrow \mathfrak{T}_1$
by
$
    \Gamma_z \, e_{j,k} = z^j \, e_k,
$
If
$s \in [0,\sigma_1(t)]$ and $x \in T$ then
$||\Gamma_{e^se_1(x)} \, t - \Gamma_{e_1(x)} \, t||_{\infty}\leq \min t/2,$
and hence for $z \in A_{\sigma_1(t)}$
\begin{equation}
\label{Gamma1}
\Re \, \Gamma_{z} \, t \geq \min t /2, \ \ \hbox{and} \ \  ||\Gamma_{z} \, t ||_{\infty} \leq ||t||_{\infty} + \min t/2.
\end{equation}
If
$z \in A_{\sigma_1(t)}$
then
$\rho(\Gamma_z \, t) \geq \rho(t)/2e$ and $\tau(\Gamma_z \, t) \leq \tau(t) + \log (2).$
Define $\zeta(t) = \lambda(\rho(t)/2e,\theta(t),\tau(t) + \log (2))$
where $\lambda$ is defined as in Theorem \ref{Psibound2}. Then this theorem
and the fact that $n(\Gamma_z \, t) = n_2(t)$ give
\begin{equation}
\label{Gamma2}
    ||\, \Psi^{\pm} \, ( \, \Gamma_z \, t \, ) \, ||_{\infty} < \zeta(t) \, n_2(t)^{\pi/2}.
\end{equation}
Inequality \ref{Gamma2}, the identity
$S^{\pm}(t)(x,y) = (\Psi^{\pm}\, (\, \Gamma_{e_1(x)} \, t \, ) \, )(y),$
and the argument used to prove Lemma \ref{decay},
generalized by considering $\Psi^{\pm} \, \Gamma_z \, t$ to be a function on $A_{\sigma_1(t)}$ with values in
the normed subspace $\mathfrak{T}_1 \subset C(T),$ gives
\begin{equation}
\label{Sbound}
    ||S^{\pm}(t) - S_N^{\pm}(t)||_{\infty} \leq 2 \zeta(t) \, n_2(t)^{\pi/2} \, \sigma_1(t)^{-1} \, e^{-N\sigma_1(t)}.
\end{equation}
\begin{thm}
\label{Sconv}
For $t \in \mathfrak{T}_{2}^{+}$ and $\epsilon > 0$ define
$$
    N(\epsilon,t) = \frac{n_1(t)}{\rho(t)} \log \left[ \frac{2\zeta(t)n_1(t)n_2(t)^{\pi/2}}{\epsilon \rho(t)} \right].
$$
Then $N \geq N(\epsilon,t)$ implies that $||S^{\pm}(t) - S_N^{\pm}(t)||_{\infty} \leq \epsilon.$
\end{thm}
\begin{proof} This follows from (\ref{Sbound}).
\end{proof}
\section{Derivation of Main Results}
$\mathfrak{T}_{2}^{p} = \{t \in \mathfrak{T}_2:t > 0\}.$
For $F \subset \mathbb{Z}^2,$
$\mathfrak{T}_{2}(F) = \{t\in \mathfrak{T}_2:\hbox{freq}(t) \subseteq F\},$
$\mathfrak{T}_{2}^{p}(F) = \mathfrak{T}_{2}^{p} \cap \mathfrak{T}_{2}(F),$
$\mathfrak{S}_2(F) = \{|t|^2: t \in \mathfrak{T}_2(F)\},$
$\mathfrak{U}_2(F) = \mathfrak{T}_{2} \cap \hbox{cl}({\mathfrak{S}_{2}(F)})$
where cl denotes closure in $C(T^2).$ Clearly $\mathfrak{U}_2(F) \subseteq \mathfrak{T}_{2}^{p}(F-F).$
%
% page 7
%
\begin{defin}
\label{propA}
$F \subset \mathbb{Z}^2$ has \emph{Property A} if $\mathfrak{T}_{2}^{p}(F-F) = \mathfrak{U}_2(F).$
\end{defin}
\noindent For $\alpha \in \mathbb{R}$ and $\beta > 0$ define
$
    F(\alpha,\beta)=\{ (j,k) \in \mathbb{Z}^2 : |k-j\alpha| < \beta  \}.
$
\begin{thm}
\label{main1}
If $\alpha = 0$ and $\beta > 0,$ then $F(\alpha,\beta)$
has property A.
\end{thm}
\begin{proof}
$F(0,\beta) = \{(j,k) \in \mathbb{Z}^2: |k| \leq n\}$ where
$n$ is the largest integer $< \beta.$
If $t \in \mathfrak{T}_{d}^{p}(F(0,\beta)-F(0,\beta))$ then $n_2(t) \leq 2n$
so $S^{+}(t) \in C^{\omega} \bigotimes \mathfrak{T}(\{0,...,2n\})$ and $t = |S^{+}(t)|^2.$
For $N \geq 1,$ $e_{0,-n}S_{N}^{+}(t) \in \mathfrak{T}_2(F(0,\beta))$ and Theorem \ref{Sconv}
implies that
$\lim_{N \rightarrow \infty} |e_{0,-n}S_{N}^{+}(t)|^2 = t$
so $t \in \mathfrak{U}_d(F).$
\end{proof}
\noindent For $F \subseteq \mathbb{Z}^2$ define $F^{\, r} = \{\, (k,j) \, : \, (j,k) \in F \, \}.$
$F$ has property A iff $F^{\, r}$ has property A. Furthermore,
if $\alpha \neq 0$ then $F^{\, r}(\alpha,\beta) = F(1/\alpha,\beta/|\alpha|).$
Therefore, without loss of generality we may assume that $|\alpha| \leq 1.$
The modular group
$SL(2,\mathbb{Z})$
acts as a group of automorphisms on the group $\mathbb{R}^2$ by
$g(x,y) = (g_{11}x+g_{12}y,g_{21}x+g_{22}y), \ \ g \in SL(2,\mathbb{Z})$ and also,
by restriction, as a group of automorphisms of the subgroup $\mathbb{Z}^2.$
This induces an action as a group of automorphisms of the algebra $\mathfrak{T}_2$
by
$g(e_{j,k}) = e_{g_{11}j+g_{12}k,g_{21}j+g_{22}k}$
and $||g(t)||_{\infty} = ||t||_{\infty}.$ Weierstrass's theorem implies that
$\hbox{cl}(\mathfrak{T}_2) = C(T^2)$ so the action extends to $C(T^2)$
and
$g(f)(x,y) = f(g_{11}x+g_{21}y,g_{12}x+g_{22}y).$
Furthermore, for
$g \in SL(2,\mathbb{Z}),$ $t \in \mathfrak{T}_2,$ and $f \in C(T^2),$
$||\, \widehat {g(t)} \, ||_{1} = ||\, \widehat t \, ||_{1},$
$||\, g(f) \, ||_{1} = ||\, f \, ||_{1},$
and
$||\, g(f) \, ||_{\infty} = ||\, f \, ||_{\infty}.$
If $f$ is real valued then $g(f)$ is real valued and its minimum and maximum values coincide with those of $f.$ A direct computation gives
\begin{equation}
\label{g1}
    g(F(\alpha,\beta)) =
    F\left(\frac{g_{21}+\alpha g_{22}}{g_{11}+\alpha g_{12}},
    \frac{\beta}{|g_{11}+\alpha g_{12}|}\right), \ \ g_{11}+\alpha g_{12} \neq 0.
\end{equation}
\begin{thm}
\label{main2}
If $\alpha \in \mathbb{Q}$ then $F(\alpha,\beta)$ has property A.
\end{thm}
\begin{proof}
If $\alpha$ is rational and $|\alpha| \leq 1$ then there exists $g \in SL(2,\mathbb{Z})$
such that $g_{21}+\alpha g_{22} = 0,$ $|g_{21}| \leq |g_{22}|,$ and $|g_{11}| \leq |g_{12}| \leq |g_{22}|/2,$ so Equation \ref{g1} gives
$g(F(\alpha,\beta)) = F(0,\beta |g_{22}|).$
Therefore, if
$t \in \mathfrak{T}_{2}^{p}(F(\alpha,\beta)-F(\alpha,\beta))$
then
$g(t) \in \mathfrak{T}_{2}^{+}(F(0,\beta|g_{22}|)-F(0,\beta|g_{22}|))$
so Theorem \ref{main1} implies that $g(t) \in \mathfrak{U}_2(F(0,\beta|g_{22}|)),$
and hence
$t \in g^{-1}\left(\mathfrak{U}_2(F(0,\beta|g_{22}|))\right).$
Since the action is an algebra automorphism that preserves the maximum norm,
it follows that
$g^{-1}\left(\mathfrak{U}_2(F(0,\beta|g_{22}|))\right) = \mathfrak{U}_2(g^{-1}(F(0,\beta|g_{22}|)) =
\mathfrak{U}_2(F(\alpha,\beta)).$
\end{proof}
\begin{defin}
\label{amply}
$\alpha \in \mathbb{R}\backslash\mathbb{Q}$ is called \emph{amply approximable} if there exists
a sequence $g \in SL(2,\mathbb{Z})$ such that
$|g_{21}+\alpha g_{22}| \, |g_{22}| \, \log |g_{22}| \rightarrow 0.$
If $|\alpha| < 1$ then we can choose $g$ so that
$\max \{ |g_{11}|, |g_{12}|, |g_{21}|, |g_{22}| \} = |g_{22}|$
and the sequence $g$ is called $\alpha$-compatible.
\end{defin}
%
% page 8
%
\begin{rem}
Davis \cite{Davis} proved that the inequality
$$
    |g_{21} + e g_{22}| < C \log (\log (|g_{22}|)) / (|g_{22}| \log (|g_{22}|)),
$$
where $g_{12}, g_{22}$ are relatively prime integers, has an infinite number of solutions for $C > 1/2$
and only a finite number of solutions for $C < 1/2.$ Therefore $e$ is not amply approximable.
For $\alpha \in \mathbb{R}\backslash\mathbb{Q}$ the Liouville-Roth constant $\mu_0(\alpha)$
is the supremum of the set of $\mu > 0$ for which there exist infinitely many
relatively prime integer pairs $(g_{21},g_{22})$ with
$|g_{12} + \alpha g_{22}| < |g_{22}|^{1-\mu}.$
$\alpha$ is a Liouville number if $\mu_0(\alpha) = \infty.$ The set of Liouville numbers is
the countable intersection of dense open sets \cite{Oxtoby}. Clearly $\mu(\alpha) > 2$ implies
that $\alpha$ is amply approximable but is a much stronger condition. Borwein and Borwein \cite{Borwein} proved that $\mu_0(e) = 2.$ If $\alpha$ is an irrational algebraic number, then the Thue-Siegal-Roth theorem implies that $\mu(\alpha) = 2$ and Lang \cite{Lang} conjectured that $\alpha$ satisfies the following stronger condition: for any $\epsilon > 0$ there exist only a finite number of relatively prime integer pairs $(g_{21},g_{22})$ with
$|g_{12} + \alpha g_{22}| < 1/(|g_{22}|\, (\log(|g_{22}|))^{1+\epsilon}).$
Clearly if $\alpha \in \mathbb{R}\backslash\mathbb{Q}$ and $\alpha$ does not satisfy Lang's condition, then $\alpha$ is amply approximable.
\end{rem}
\noindent The proof of our main theorem depends on the following observation which relates Diophantine approximation with the geometry of the lattice $\mathbb{Z}^2.$
\begin{lem}
\label{lattice}
    If $g_{21}, g_{22}$ are relatively prime, $\widetilde \alpha = -g_{21}/g_{22},$
    $0 < |\alpha| < 1,$ $\widetilde \beta \in (0,\beta),$ $(j,k) \in F(\widetilde \alpha,\widetilde \beta),$
    and $|j| < (\beta - \widetilde \beta) \, / \, |\alpha - \widetilde \alpha|,$ then $(j,k) \in F(\alpha,\beta).$
\end{lem}
\begin{proof}
Since $|k-\widetilde \alpha \, j| < \widetilde \beta$ therefore
$|k-\alpha \, j| \leq |k-\widetilde \alpha j| + |j|\, |\alpha - \widetilde \alpha| < \beta.$
\end{proof}
\begin{thm}
\label{main3}
If $\alpha$ is amply approximable then $F(\alpha,\beta)$ has property A.
\end{thm}
\begin{proof}
Since $\alpha$ is amply approximable iff $1/\alpha$ is, we may assume that
$|\alpha| < 1.$
Assume that $\epsilon > 0$ and
$t \in \mathfrak{T}_{2}^{p}(F(\alpha,\beta)-F(\alpha,\beta)).$
Choose $\widetilde \beta \in (0,\beta)$ with
$t \in \mathfrak{T}_{2}^{p}(F(\alpha,\widetilde \beta)-F(\alpha,\widetilde \beta))$
and choose an $\alpha$-compatible sequence $g$ in $SL(2,\mathbb{Z}).$
Construct sequences $\widetilde \alpha(g) = -g_{21}/g_{22}$ and
$n(g) = $ largest integer $< \widetilde \beta |g_{22}|.$
Since $\widetilde \alpha(g) \rightarrow \alpha$ we may assume
$t \in \mathfrak{T}_{2}^{p}(F(\widetilde \alpha(g),\widetilde \beta)-F(\widetilde \alpha(g),\widetilde \beta)),$
and hence
$g(t) \in \mathfrak{T}_{2}^{p}(F(0,\widetilde \beta \, |g_{22}|)-F(0,\widetilde \beta |g_{22}|)).$
Construct sequences $s(g)$ in $\mathfrak{T}_{2}$ and $a(g)$ in $\mathbb{Z}_{+}$ by
\begin{equation}
\label{sg}
    s(g) = g^{-1}\left( e_{0,-n(g)} \, S_{N(\epsilon,g(t))}^{+}(g(t)) \right),
\end{equation}
\begin{equation}
\label{ag}
    a(g) = n_1(s(g))/(|g_{22}|^2 \log |g_{22}|).
\end{equation}
Then $s(g) \in \mathfrak{T}_2(F(\widetilde \alpha(g),\widetilde \beta)).$
Since
$\rho(g(t)) = \rho(t)$ and $\zeta(g(t)) = \zeta(t),$
$N(\epsilon,g(t)) \approx O(|g_{22}| \log (|g_{22}|)),$
and (\ref{sg}) gives
$n_1(s(g)) \approx O(|g_{22}|\, N(\epsilon,g(t)).$
%
% page 9
%
Therefore
$n_1(s(g)) \approx O(|g_{22}|^2 \log (|g_{22}|)),$
so $a(g)$ is bounded.
Since $g$ is $\alpha$-compatible, for sufficiently large $|g_{22}|,$
$n_1(s(g)) < (\beta - \widetilde \beta)/|\alpha - \widetilde \alpha(g)|.$
Therefore Lemma \ref{lattice} implies that
$s(g) \in \mathfrak{T}_2(F(\alpha,\beta)).$
Theorem \ref{Sconv} implies that
$|| \, t - |s(g)|^2 \, ||_{\infty} \leq (2||t||_{\infty} + \epsilon)\, \epsilon.
$
Since $\epsilon$ was arbitrarily small it follows that
$t \in \mathfrak{U}_2(F(\alpha,\beta))$
so $F(\alpha,\beta)$ has property A.
\end{proof}
\begin{rem}
A result of Erd\"{o}s and Turan \cite{Erdos} implies that the roots of the Laurent polynomials associated to $g(t)$ become more equidistributed near $T_c$ as $|g_{22}|$ increases. Then results of Amoroso and Mignotte \cite{Amoroso} might be used to sharpen the upper bound for $||\Psi^{+}(g(t))||_{\infty}.$ This would show that $S_{N}^{+}(g(t))$ converges faster that shown in this paper and possibly give a proof that $F(\alpha,\beta)$ has property A under
weaker conditions on $\alpha$ than the amply approximable condition.
\end{rem}

Address: Wayne Lawton, 135/125 Laksi, Bangkok 10210, Thailand

\end{document}